\documentclass[a4paper,11pt]{amsart}
\usepackage[utf8]{inputenc}
\usepackage{csquotes}
\usepackage{amsmath, amssymb, amsthm, xcolor,enumerate, enumitem}
\usepackage[english]{babel}
\usepackage{tikz-cd}
\usepackage[all]{xy}
\usepackage{microtype}
\usepackage[colorlinks=true, citecolor=blue, linkcolor=blue, bookmarks=true]{hyperref}
\usepackage{bbm}
\usepackage{graphicx}
\usepackage{adjustbox}
\usepackage{comment}
\usepackage{hyperref}
\usepackage{stmaryrd}
\usepackage{mathtools}

\newcounter{counter}

\newtheorem{question}{Question}
\newtheorem{theorem}[counter]{Theorem}
\newtheorem{lemma}[counter]{Lemma}

\newtheorem{definition/proposition}[counter]{Definition/Proposition}
\newtheorem{corollary}[counter]{Corollary}
\newtheorem{proposition}[counter]{Proposition}
\newtheorem*{theorem*}{Theorem}
\numberwithin{equation}{section}

\theoremstyle{definition}
\newtheorem{definition}[counter]{Definition}

\newtheorem{remark}[counter]{Remark}

\newcommand{\ti}{\widetilde}
\newcommand{\tu}{\widetilde{u}}

\newcommand{\cA}{\mathcal{A}}

\newcommand{\cP}{\mathcal{P}}
\newcommand{\cO}{\mathcal{O}}
\newcommand{\cV}{\mathcal{V}}
\newcommand{\cU}{\mathcal{U}}

\newcommand{\cZ}{\mathcal{Z}}
\newcommand{\cH}{\mathcal{H}}
\newcommand{\cK}{\mathcal{K}}
\newcommand{\cR}{\mathcal{R}}

\newcommand{\bR}{\mathbb{R}}
\newcommand{\bC}{\mathbb{C}}
\newcommand{\bT}{\mathbb{T}}
\newcommand{\bN}{\mathbb{N}}

\newcommand{\bZ}{\mathbb{Z}}

\newcommand{\ob}{\mathrm{ob}}
\newcommand{\Ad}{\mathrm{Ad}}

\newcommand{\Hom}{\operatorname{Hom}}

\title[Projective representations]{Projective representations on operator algebras}

\author{Sergio Girón Pacheco}
\address{\hskip-\parindent Sergio Girón Pacheco, Department of Mathematics, KU Leuven, Celestijnenlaan 200B, 3001, Leuven, Belgium.}
\email{sergio.gironpacheco@kuleuven.be}

\begin{document}

\begin{abstract}
We establish new restrictions for the values of the 2 cohomology lifting obstruction for projective unitary representations of second countable, locally compact Hausdorff groups on operator algebras. Using these, we show that every projective representation on the Jiang--Su algebra lifts to a genuine group representation. We also characterise the possible values of lifting obstructions for finite group projective representations on UHF-algebras of infinite type and Cuntz algebras. Finally, we show that certain 2-cocycles of Property (T) groups cannot arise as lifting invariants of projective representations on the hyperfinite II$_{1}$ factor.
\end{abstract}
\maketitle
\numberwithin{counter}{section}
\section*{Introduction}
A projective unitary representation of a separable locally compact group $G$, is a strongly continuous homomorphism from $G$ to the projective unitary group of a Hilbert space. These first arose in the work of Schur for discrete groups, where they were shown to be closely related to representations of certain central extensions \cite{SCH04}. The theory was later extended to topological groups in the work of Mackey \cite{MA58}. 
\par Projective representations rose to prominence in quantum physics in work of Wigner.\ This is due to the fact that states are only defined up to a phase, so projective representations are a more natural object of study than genuine group representations (see \cite{TA22} for a comprehensive explanation of this matter). Of particular interest are the irreducible projective representations of the Poincare group, the group of symmetries of space-time, which were classified in \cite{WI39}. Note also that the projective representation theory of many other groups is also important in physics; an example of this can be found in \cite{RevB}, see also \cite{Projsurvey} for a recent review on the relevance of finite group projective representations in quantum many body systems.
\par It is natural to ask when a projective representation lifts to a genuine unitary representation. This was originally studied very successfully through ad hoc methods in \cite{BA54}. This question is now understood as a problem in the realm of group cohomology. Indeed, the obstruction to lifting is precisely given by a class $\omega\in H^2_b(G,\bT)$, where $H^n_b$ is the \emph{Borel cohomology} introduced by Moore in \cite{MO76} (see Section \ref{subsec:projreps}).
\par The theory of projective representations admits a straighforward generalisation to the setting of separable, unital C$^*$-algebras and von Neumann algebras. In this case, a projective representation on a C$^*$-algebra/von Neumann algebra $A$ is a continuous homomorphism $G\rightarrow P\cU(A)$ where $\cU(A)$ is the unitary group of $A$ equipped with the norm/$\sigma$-strong topology and $P\cU(A)$ is its quotient by its centre. Any such projective representation also admits a lifting obstruction $\omega\in H^2_b(G,Z\cU(A))$. This article is concerned with the following question.
\begin{question}\label{introques: liftingobs}
    For a group $G$ and a C$^*$ or von Neumann algebra $A$ what are the possible values of lifting obstructions for projective representations of $G$ on $A$?
\end{question}
\par In the case of an infinite-dimensional factor $M$, and a discrete amenable group $G$, it is a consequence of Connes and Haagerup's classification results (\cite{CO76},\cite{HAA87}) that every value in $H^2(G,\bT)$ arises as the lifting obstruction of a projective representation on $M$. Indeed, for any $[\omega]\in H^2(G,\bT)$ the twisted group von Neumann algebra $L(G,\omega)$ is injective. Thus, as a consequence of Connes' theorem, there is a trace preserving unital embedding of $L(G,\omega)$ into the hyperfinite II$_{1}$ factor $\cR$. The image of the unitary generators of $L(G,\omega)$ defines a unitary projective representation on $\cR$ with lifting obstruction $\omega$. The case for an arbitrary factor $M$ now follows since $\cR$ admits a unital embedding into $M$.
\par Analogously one may consider Question \ref{introques: liftingobs} in the case of infinite-dimensional, unital C$^*$-algebra $A$ with trivial centre. Here we show that the solution varies greatly depending on the K-theory of $A$. To do this we draw motivation from the invariants for $G$-kernels introduced in the articles \cite{EVGI23} and \cite{IZ23}. Using K-theory, we introduce new $2$-cohomology invariants of projective representations on C$^*$-algebras. By studying the relation of these invariants with the lifting obstruction, we obtain restrictions to its possible values. In particular, we show that for the infinite dimensional simple amenable C$^*$-algebra $\cZ$, introduced in \cite{Jiang-Su}, all projective representations lift to genuine representations.
\begin{theorem*}[{cf. Theorem \ref{thm:JiangSu}}]
    Every projective representation of a separable locally compact group $G$ on the Jiang--Su algebra $\cZ$ lifts to a genuine representation of $G$.
\end{theorem*}
\par The Jiang--Su algebra $\cZ$ plays a central role in the Elliott classification programme for amenable C$^*$-algebras (see \cite{WH23} for a survey on this programme). Indeed, those C$^*$-algebras that fall within the scope of the classification theorem, obtained by combining \cite{KI95, phillipsclass, GLN1, ELGOLINI25, TWW17}, admit a unital embedding of $\cZ$ (\cite{WI12}). We call C$^*$-algebras that fall within the scope of the classification theorem \emph{classifiable}. Therefore, if the Jiang--Su algebra were to admit a non-trivial $2$-cocycle invariant for a projective representation, so would any other classifiable C$^*$-algebra. The theorem above implies that Question \ref{introques: liftingobs} is of particular interest for classifiable C$^*$-algebras, and that the range of $2$-cocycle invariants for projective representations of any group may vary greatly for C$^*$-algebras within this class (indeed we show that this is the case, see for example Remark \ref{rmk:O2}). We thus proceed by studying other prominent C$^*$-algebras within the scope of the classification theorem.
\par It is worth noting that the argument that gives rise to the theorem above makes use of the de la Harpe--Skandalis determinant (\cite{SkdlH}). This argument is closely related to \cite[Theorem 4]{Projsurvey} which restricts the order of lifting obstructions of a projective representation of a given finite dimension through the classical determinant.
\par We then perform a systematic analysis of the possible values of $2$-cocycle invariants for finite group projective representations on both UHF-algebras of infinite type and the Cuntz-algebras introduced in \cite{CU77}. Combining the restrictions to the lifting invariant of K-theoretic nature with specific constructions of projective representations, we obtain a characterisation of the possible values of lifting obstruction in these cases. Observe that the analogous question in the case of the (3-cohomological) lifting obstruction for $G$-kernels is still open \cite[Conjecture 2.9]{IZ23}.
\begin{theorem*}[{cf. Theorem \ref{thm:UHF} and Theorem \ref{thm:cuntz}}]
    Let $G$ be a finite group and $[\omega]\in H^2(G,\bT)$ be non-trivial with its order denoted by $r>1$. Then $\omega$ arises as the lifting obstruction of a projective representation on
    \begin{itemize}
        \item $M_{n^\infty}$ for $n\geq 2$ if and only if $r^\infty$ formally divides $n^\infty$.
        \item $\cO_{n+1}$ for $1\leq n\leq \infty$ if and only if $n$ is finite and $[\omega]\in nH^2(G,\bT)$.
    \end{itemize}
\end{theorem*}
For a discrete group $G$ and $[\omega]\in H^2(G,\bT)$, there exists an $\omega$-projective representation on a unital C$^*$-algebra $A$ with trivial centre if and only if there is a unital $^*$-homomorphism from the full $\omega$-twisted C$^*$-algebra of $G$ into $A$. This change in perspective clarifies why K-theory should play a role in Question \ref{introques: liftingobs}. But also, it allows the use of existence of embedding results in the spirit of \cite[Corollary C]{classification1}, to construct projective representations with a prescribed $2$-cocycle invariant. We exploit these techniques in several cases to study the existence question for projective representations of infinite groups (see Remark \ref{rmk:infdiscreteUHF}, Proposition \ref{prop: Oinftgrp} and Remark \ref{rmk:O2}). 
\par As every $2$-cocycle of a discrete amenable group $G$ arises as the lifting obstruction of a projective representation on $\cR$, it is natural to ask whether the same holds for arbitrary discrete groups. In the final section of this paper we show that this is not the case. We prove that there are $2$-cocycles on certain residually finite property (T) groups which do not arise as lifting obstructions of projective representations on $\cR$.
\subsection*{Acknowledgements}
The author would like to thank Stefaan Vaes, whose insights were crucial to the development of this paper.\ Also, he would like to thank Fran\c cois Thilmany for discussions on extensions of residually finite groups. The author was funded by project 1249225N of the Research Foundation Flanders (FWO).
\section{Preliminaries}
Throughout this paper $G$ will denote a locally compact Hausdorff second countable group and $A$ will be a unital, separable C$^*$-algebra. We denote by $A^{sa}$ its real subspace of self-adjoint elements and by $T(A)$ its space of tracial states. We let $M$ be a von Neumann algebra with separable predual. We will denote by $\cU(A)$ the unitary group of $A$ equipped with the norm topology and by $\cU(M)$ the unitary group of $M$ equipped with the $\sigma$-strong topology. Note that under these topologies both $\cU(A)$ and $\cU(M)$ are Polish groups. We denote by $P\cU(A)$ and $P\cU(M)$ the respective quotients by their centres. We will write $\cU(n)$ for the unitary group of $M_n(\bC)$. In the case of a Hilbert space $\cH$ we will denote the group of unitary operators by $\cU(\cH)$ which we equip with the strong operator topology. Let $G,H$ and $Q$ be topological groups. Following \cite{MO76} we will say that
\[
H\xrightarrow{\iota}G\xrightarrow{p}Q
\]
is a short exact sequence of topological groups when the sequence is algebraically exact, $\iota$ and $p$ are continuous group homomorphisms and $p:G\rightarrow Q$ is open. This implies that the canonical group homomorphism $G/H\rightarrow Q$ is a homeomorphism.
\par Throughout this paper we will often use the following two well known results. 
\begin{lemma}[{cf. \cite[Proposition 4]{MO76}}]\label{lem:lift}
    Let $G$ be a Polish group with a closed subgroup $H$ and $\pi:G\rightarrow G/H$ be the projection. Then there exists a Borel section $s:G/H\rightarrow G$.
\end{lemma}
%\begin{lemma}\label{lem:lift}
 %   Let $G$ and $\Gamma$ be Polish groups and $H$ a closed subgroup of $G$. Then any continuous homomorphism $\phi:\Gamma\rightarrow G/H$ admits a Borel section $\widetilde{\phi}:\Gamma\rightarrow G$.
%\end{lemma}
%\begin{proof}
%Denote the projection $\pi:\Gamma\rightarrow G/H$. Consider the closed subgroup $P\subset \Gamma\times G$ defined by $P=\{(\gamma,g):\phi(\gamma)=\pi(g)\}$. As $P$ is a closed subgroup of a Polish group it is itself Polish. A Borel transversal with respect to the equivalence relation given by the closed subgroup $\{1\}\times H$ of $P$ exists by \cite[Theorem 12.17]{KE95}. This yields the desired Borel lift. 
%\end{proof}
\begin{lemma}[{cf. \cite[Proposition 5]{MO76}}]\label{lem:automaticcontinuity}
    Let $G_1$ and $G_2$ be Polish groups then any Borel homomorphism $\varphi: G_1\rightarrow G_2$ is continuous.
\end{lemma}
For a group $G$ and a continuous $G$-module $N$ we will denote by $Z^n_b(G,N)$ the Borel $n$-cocycles of $G$ with values in $N$. We will denote the associated cohomology groups by $H^n_b(G,N)$ as introduced by Moore in \cite{MO76}. In our setting the coefficient module $N$ will always be a Polish abelian group. In this case $H_b^n(G,\cdot)$ is a functor of the coefficient module that turns short exact sequences into long exact sequences. Along with another couple of natural conditions, $H_b^n$ is the unique such functor (see \cite[Theorem 2]{MO76}). When $G$ is discrete $H_b^n(G,N)$ coincides with the usual cohomology for discrete groups (of \cite{BRO82} for example), in this case we simply denote it by $H^n(G,N)$.
\par For any discrete group $G$ and $\omega\in Z^2(G,\bT)$ we can form the twisted group algebra $\bC[G,\omega]$ as the universal $^*$-algebra generated by unitaries $u_g$ for $g\in G$ satisfying the relation
\[
u_gu_h=\omega(g,h)u_{gh}.
\]
It is easy to see that the isomorphism class of $\bC[G,\omega]$ as a $^*$-algebra coincides for cohomologous $2$-cocycles. One can represent $\bC[G,\omega]$ as operators on $l^2(G)$ under the $\omega$-regular representation with
\[u_g^{\omega}\delta_h=\omega(g,h)\delta_{gh}\]
for all $g,h\in G$. The associated reduced and full C$^*$-algebras are denoted by $C_r^*(G,\omega)$ and $C^*(G,\omega)$ respectively. 
\subsection{Projective unitary representations} \label{subsec:projreps}
A projective unitary representation of a group $G$ is a continuous homomorphism $G\rightarrow P\cU(\cH)$ (where the target is equipped with the quotient of the strong or equivalently $\sigma$-strong topology).\ There is a $2$-cohomology obstruction to lifting a projective unitary representation to a unitary representation (see \cite{MA58}). In this section we discuss the analogous obstruction in the generality of projective representations on C$^*$ and von Neumann algebras.
\par A \emph{projective unitary representation} of a group $G$ on a C$^*$-algebra $A$, or simply a projective representation, is a continuous group homomorphism $G\rightarrow P\cU(A)$. One can define projective representations on a von Neumann algebra $M$ in the same way. For the proceeding discussion we only comment on the case of C$^*$-algebras, although the case of von Neumann algebras is analogous.
\par   For any projective representation $u:G\rightarrow P\cU(A)$ one can choose a Borel lifting $\widetilde{u}:G\rightarrow \cU(A)$ as an application of Lemma \ref{lem:lift}. Note that, equipping $Z\cU(A)$ with the trivial $G$-module structure, the function $\lambda_{g,h}=\widetilde{u}_g\widetilde{u}_h\widetilde{u}_{gh}^*\in Z^2_b(G,Z\cU(A))$ by the following computation
\begin{align*}
\lambda_{h,k}\lambda_{g,hk}\tu_{ghk}&=\lambda_{h,k}\widetilde{u}_g\widetilde{u}_{hk}\\
&=\tu_g\lambda_{h,k}\tu_{hk}=\widetilde{u}_g\widetilde{u}_h\widetilde{u}_k=\lambda_{g,h}\widetilde{u}_{gh}\widetilde{u}_k=\lambda_{g,h}\lambda_{gh,k}\tu_{ghk}.
\end{align*}
Moreover, for any other Borel lift $\tu'$ with associated Borel $2$-cocycle $\lambda'$ it follows from a direct computation that $v_g=\tu'_g\tu_g^*\in Z\cU(A)$ satisfies $\lambda_{g,h}'=(\partial v)_{g,h}\lambda_{g,h}$. In particular $\lambda'$ and $\lambda$ are Borel cohomologous and the class $[\lambda_{g,h}]\in H^2_b(G,Z\cU(A))$ is independent of the choice of Borel lift of $u$.
\begin{definition}
    Let $u:G\rightarrow P\cU(A)$ be a projective representation. We denote the associated class in $H^2(G,Z\cU(A))$ by $\ob(u)$.
\end{definition}
We note that $\ob(u)$ is precisely the obstruction to lifting $u$ to a continuous representation $G\rightarrow \cU(A)$. Indeed, if $\ob(u)=[\lambda]=0$, then $\lambda$ is a coboundary amd thus there exists a Borel lift $\tu:G\rightarrow \cU(A)$ and a Borel function $\mu:G\rightarrow Z\cU(A)$ such that $\lambda_{g,h}=\mu_g\mu_h\overline{\mu_{gh}}$. Thus the function $g\mapsto \ti{u}_g\overline{\mu_g}$ gives a Borel group homomorphism lift of $u$. The group $G$ is second countable locally compact Hausdorff group $G$ and thus Polish by \cite[Theorem 5.3]{KE95}, as $\cU(A)$ Polish, the homomorphism $\ti{u}\overline{\mu}$ is continuous by Lemma \ref{lem:automaticcontinuity}. 
%\cite[Theorem 9.10]{KE95}

\par For a group $G$, $[\omega]\in H^2_b(G,\bT)$ and a C$^*$-algebra $A$ we will often call a projective representation $u:G\rightarrow P\cU(A)$ with $\ob(u)=[\omega]$ a $(G,[\omega])$-representation on $A$.

\section{projective representations on C*-algebras}\label{sec:projrepCstar}
As in the case of lifting obstructions for $G$-kernels (see \cite{thesis:sergio,EVGI23,IZ23}) there are K-theoretic restrictions to the values of lifting obstructions for projective representations of C$^*$-algebras too. We will develop these in this section. 
\par Let $\tau\in T(A)$ and denote by $\Delta_\tau:\cU(A)_0\rightarrow \bR/\tau(K_0(A))$ the Skandalis de la Harpe determinant of \cite{SkdlH}. Recall that $\Delta_\tau$ is a group homomorphism and satisfies
\begin{equation}\label{eqn:SdlHexponentials}
\Delta_\tau(e^{i h})=\frac{1}{2\pi}\tau(h)
\end{equation}
for $h\in A^{sa}$. It is shown in \cite[Proposition 2.7]{GIIZPE25} (see also \cite{ROCH23}) that whenever the canonical map $\pi_1(\cU(A))\rightarrow K_0(A)$ is a surjection, then the kernel of $\Delta_\tau$ coincides with the unitaries of the form
\[
\ker(\Delta_\tau)=\{u \in \cU(A): u=e^{ih_1}e^{ih_2}\ldots e^{ih_n},\ h_i\in A^{sa}\cap \ker(\tau), n\in \bN\}.
\]
Moreover, when equipped with the metric
\[
d_\tau(u,v)=\inf\{\sum_{j=1}^n\|h_j\|:h_j\in A^{sa}\cap\ker(\tau), uv^*=e^{ih_1}\ldots e^{ih_n}\},
\]
the kernel of $\Delta_\tau$ is a Polish group whenever $A$ is separable. We denote by $S\cU_\tau(A)$ the kernel of $\Delta_\tau$ equipped with the $d_\tau$ topology.
For the remainder of this discussion assume that $\cU(A)$ is connected. Note that any $u\in \cU(A)$ can be written as a product $u=vz$ for $v\in S\cU_\tau(A)$ and $z\in \bT$. Indeed
\[
ue^{-2\pi i\overline{\Delta_\tau(u)}}\in \ker(\Delta_\tau).
\]
where $\overline{\Delta_\tau(u)}$ denotes a lifting of $\Delta_\tau(u)\in \bR/\tau(K_0(A))$ to $\bR$. In particular, $ZS\cU_\tau(A)$ agrees with $Z\cU(A)\cap S\cU_\tau(A)$ as sets, and the canonical mapping
\[
\varphi:S\cU_\tau(A)/ZS\cU_\tau(A)\rightarrow P\cU(A)
\]
is an isomorphism of sets. Moreover, by \cite[Lemma 2.9]{GIIZPE25} the mapping $\varphi$ is in fact an isomorphism of topological groups. Thus, for any projective representation $u:G\rightarrow P\cU(A)$, one can apply Lemma \ref{lem:lift} to lift to a Borel map $\widetilde{u}\in S\cU_\tau(A)$. Now, as in Section \ref{subsec:projreps}, $\lambda_{g,h}=\ti{u}_g\ti{u}_h\ti{u}_{gh}^* \in ZS\cU_\tau(A)$  is a Borel $2$-cocycle and its cohomology class is well-defined.
\begin{definition}\label{def:obtau}
Let $A$ be a unital, separable C$^*$-algebra with a connected unitary group such that the canonical map $\pi_1(\cU(A))\rightarrow K_0(A)$ is a surjection. Let $u:G\rightarrow P\cU(A)$ be a projective representation and $\tau\in T(A)$ a trace. We denote the associated class in $H_b^2(G,ZS\cU_\tau(A))$ by $\ob_\tau(u)$.
\end{definition}
As discussed in \cite{GIIZPE25}, the condition that the canonical map $\pi_1(\cU(A))\rightarrow K_0(A)$ is a surjection holds in many cases of interest. For example, it holds when $A$ absorbs the Jiang--Su algebra of \cite{Jiang-Su}.
\begin{remark}
    We will often care about the case when $A$ has trivial centre. Under the other standing assumptions of Definition \ref{def:obtau} we have that $ZS\cU_\tau(A)=\bT\cap \ker(\Delta_\tau)=e^{2\pi i \tau(K_0(A))}$ as a set. Moreover, by \cite[Remark 2.8]{GIIZPE25} $ZS\cU_\tau(A)$ is equipped with the discrete topology. Thus, in this case, we may and will identify $ZS\cU_\tau(A)$ with $\tau(K_0(A))/\bZ$.
\end{remark}
As the $d_\tau$ topology is stronger than the norm topology, the inclusion $\iota:S\cU_\tau(A)\hookrightarrow \cU(A)$ is continuous. Hence any Borel lift $\widetilde{u}$ to $S\cU_\tau(A)$ also gives a Borel lift $\iota \circ\widetilde{u}$ to $\cU(A)$ and so 
\begin{equation}\label{eqn: inclusion}
    \iota_*(\ob_\tau(u))=\ob(u).
\end{equation}
where $\iota_*:H^2_b(G,ZS\cU_\tau(A))\rightarrow H^2_b(G,Z\cU(A))$ is the group homomorphism induced by the map $\iota$ of coefficient modules. This immediately gives the following result.
\begin{theorem}\label{thm:JiangSu}
    Let $\cZ$ be the Jiang-Su algebra and $G$ a locally compact second countable Hausdorff group. Then any projective unitary representation of $G$ on $\cZ$ lifts to a unitary representation.
\end{theorem}
\begin{proof}
    Let $u:G\rightarrow P\cU(\cZ)$ be a projective representation. By \cite{Jiang-Su} the unitary group of $\cZ$ is connected and its $K_0$ is $\bZ$ with the unique trace inducing the identity on $\bZ$. Moreover $\cZ$ is itself $\cZ$-stable so $\pi_1(\cU(\cZ))\rightarrow K_0(\cZ)$ is a surjection. Thus $\ob_\tau(u)$ is valued in the trivial group and so itself trivial. By \eqref{eqn: inclusion} the result follows.
\end{proof}
\begin{remark}
    The proof of Theorem \ref{thm:JiangSu} gives that any unital, separable C$^*$-algebra $A$ with a trace $\tau$ such that
    \begin{itemize}
        \item $A$ has a connected unitary group,
        \item the map $\pi_1(\cU(A))\rightarrow K_0(A)$ is surjective,
        \item $Z(A)=\bC$,
        \item $\tau(K_0(A))=\bZ$,
    \end{itemize}
    will satisfy the property that any projective representation of a locally compact second countable Hausdorff group $G$ on $A$ lifts to a unitary representation of $G$ on $A$. For example, by Elliott's range of the invariant result \cite{ELL96} (see also \cite[Theorem 5.3]{GOLI22}) there is a simple, unital, separable, $\cZ$-stable C$^*$-algebra $A$ with a trace $\tau\in T(A)$ with $K_1(A)=0$, $K_0(A)\cong \bZ^2$ and the induced map by $\tau:\bZ^2\rightarrow \bR$ is given by $\tau(n,m)=n$. Then by construction $A$ satisfies all the bullet points above and thus any projective representation of $G$ on $A$ lifts to a unitary representation.
\end{remark}
We can also characterise, for finite groups $G$, the values of $2$-cocycles that arise on UHF-algebras of infinite type. Before we do so we need two intermediate results. In the following proof we denote the matrix units of $M_n$ by $e_{i,j}$.
\begin{lemma}\label{lem:projrepsfd}
    Let $G$ be a finite group and $[\omega]\in H^2(G,\bT)$ of order $n$. Then there exists an $[\omega]$-projective representation $u:G\rightarrow P\cU(n^k)$ for some $k\in \bN$. In fact, $k$ can be chosen to be $2\lvert G\rvert-1$.
\end{lemma}
\begin{proof}
    By the long exact sequence in group cohomology associated to the short exact sequence of coefficient modules
    \[
    \bZ_n\xrightarrow{\iota}\bT\xrightarrow{\times n}\bT,
    \]
    $\omega$ is of order $n$ if and only if it comes from the image of the map $\iota_*:H^2(G,\bZ_n)\rightarrow H^2(G,\bT)$. Thus we may write $\omega(g,h)=\xi^{\lambda(g,h)}$ where $\lambda(g,h)\in \bZ$ satisfies the $2$-cocycle identity mod $n$ and $\xi$ is a primitive $n$-th root of unity.
    \par Let $\cK=\bigotimes_{k\in G}M_n$ which is a Hilbert space of dimension $n^{2|G|}$ with orthonormal basis of the form $\bigotimes_{k\in G}e_{i_{k},j_{k}}$ for $1\leq i_k,j_k\leq n$. Set
    $U_0=\sum_{i=1}^n\xi^{i}e_{i,i}\in M_n$ and $U=\bigotimes_{k\in G}U_0$. Let $\sigma_g$ be the shift by $g$ on the tensor product which defines an element of $\cU(\cK)$ and let
    \[
    \pi(g)=\sigma_g\circ \bigotimes_{k\in G}\Ad(U_0^{\lambda(g,k)})\in \cU(\cK)
    \]
    Note that, as $\lambda$ satisfies the $2$-cocycle identity mod $n$, one has that
    \begin{align}\label{eqn:pirelation}
        \pi(g)\pi(h)&=\sigma_g\circ \bigotimes_{k\in G}\Ad(U_0^{\lambda(g,k)})\circ \sigma_h\circ \bigotimes_{k\in G} \Ad(U_0^{\lambda(h,k)})\notag\\
        &=\sigma_{gh}\circ \bigotimes_{k\in G} \Ad(U_0^{\lambda(g,hk)+\lambda(h,k)})\notag\\
        &=\sigma_{gh}\circ \bigotimes_{k\in G}\Ad(U_0^{\lambda(g,h)}U_0^{\lambda(gh,k)})\\
        &=\Ad (U^{\lambda(g,h)})\circ \pi(gh).\notag
    \end{align}
    Consider the subspaces of $\cK$ given by 
    \[
    \cK_{l}=\{A\in \cK: UAU^*=\xi^l A\}
    \]
    for $1\leq l\leq n$ and let 
    \[
    S=(e_{1,n}\otimes\sum_{i=2}^n e_{i,i-1})\bigotimes_{G\setminus 1_G} I_n.
    \]
    Then $\cK\cong \oplus_{l=1}^n \cK_{l}$ and the map
    \[
    L_S:A\in\cK\mapsto SA
    \]
    induces an isomorphism from $\cK_l$ to $\cK_{l+1}$ for all $1\leq k \leq n$ where we count indices modulo $n$ (i.e. letting $\cK_{n+1}=\cK_1$). Indeed, if $A\in \cK_l$ then
    \[
    S^*USAU^*=\xi UAU^*=\xi^{l+1} A
    \]
    and so $L_S$ maps $\cK_l$ to $\cK_{l+1}$. As the inverse of $L_S$, which is given by $L_{S^*}$, maps $\cK_{l+1}$ to $\cK_l$ it follows that $\cK_l\cong \cK_{l+1}$ as claimed. Thus each $\cK_l$ has dimension $n^{2|G|-1}$. Note that $\Ad(U)$ commutes with $\pi(g)$ and so $\pi(g)$ preserves each $\cK_l$. Now set
    \[
    \phi(g)=\pi(g)|_{\cK_1}.
    \]
    The operators $\phi(g)$ are unitary and by \eqref{eqn:pirelation} satisfy
    \[
    \phi(g)\phi(h)=\xi^{\lambda(g,h)}\phi(gh)=\omega(g,h)\phi(gh)
    \]
    which gives a $(G,[\omega])$-projective representation of dimension $n^{2|G|-1}$ .
\end{proof}
\begin{proposition}\label{prop:UHFinfinite}
    Let $G$ be a discrete group such that $H_2(G,\bZ)$ is finitely generated and $[\lambda]\in H^2(G,\bT)$. If there exists a $(G,[\lambda])$-representation on $M_{n^\infty}$ for $n\geq 2$ then  $n^k[\lambda]=0$ for some $k\in \bN$.
\end{proposition}
\begin{proof}
    By the universal coefficient theorem for group cohomology \cite[Proposition 11.9.2]{TD08} one has a natural isomorphism
    \begin{equation}\label{eqn:UCTidentification}
        H^2(G,\bT)\cong \Hom_{\bZ}(H_2(G,\bZ),\bT).
    \end{equation}
    If $u:G\rightarrow P\cU(M_{n^\infty})$ is a projective unitary representation then $\ob(u)=\iota_*(\ob_\tau(u))$ by \eqref{eqn: inclusion}. So under the identification \eqref{eqn:UCTidentification} we have that 
    \[
    \ob(u)\in \Hom_{\bZ}(H_2(G,\bZ),\bZ[1/n]/\bZ).
    \]
    Any element of $\bZ[1/n]/\bZ$ is annihilated by $n^k$ for some $k\in \bN$, so as $H_2(G,\bZ)$ is finitely generated, $\ob(u)$ is also annihilated by $n^k$ for some $k\in \bN$. 
\end{proof}
We can now characterise the $2$-cocycles that arise from projective representations of finite groups on UHF-algebras of infinite type.
\begin{theorem}\label{thm:UHF}
    Let $n\in \bN$ with $n\geq2$, $G$ be a finite group and $[\lambda]\in H^2(G,\bT)$. Then there exists a $(G,[\lambda])$-representation on $M_{n^\infty}$ if and only if $n^k[\lambda]=0$ for some $k\in \bN$.
\end{theorem}
\begin{proof}
    If $n^k[\lambda]=1$ for some $k\in \bN$ then $\lambda$ has order $r$ dividing $n^k$. By Lemma \ref{lem:projrepsfd} there exists a projective representation $v:G\rightarrow P\cU(l)$ with $l$ a power of $r$ and $\ob(v)=[\lambda]$. As there exists a unital embedding $M_l\rightarrow M_{r^\infty}\rightarrow  M_{n^\infty}$  then one also has a projective representation $u:G\rightarrow P\cU(M_{n^\infty})$ with $\ob(u)=[\lambda]$. The converse follows from Proposition \ref{prop:UHFinfinite} as $H_2(G,\bZ)$ is a finite group when $G$ is finite.
\end{proof}
\begin{remark}\label{rmk:infdiscreteUHF}
   For $G=\bZ^2$ the obstruction of Proposition \ref{prop:UHFinfinite} is the unique obstruction to $(G,[\omega])$-representations. This follows by using C$^*$-classification techniques. In this case $[\omega]\in H^2(\bZ^2,\bT)\cong \bT$. The existence of a $(G,[\omega])$ representation on $M_{n^\infty}$ is equivalent to a $^*$-homomorphism
    \[
    C^*(\bZ^2,\omega)\cong \cA_{\theta}\longrightarrow M_{n^\infty}
    \]
    for $\theta\in [0,1]$ such that $e^{2\pi i\theta}=[\omega]$ with $\cA_{\theta}$ denoting the rotation algebra. By \cite[Corollary C]{classification1} there exists such a $^*$-homomorphism whenever $\theta=k/p_1^{r_1}p_2^{r_2}\ldots p_n^{r_n}$ for $k\in \bN$,  $r_i\in \bN$ and $p_i$ are primes dividing $n$. Indeed, this follows as the rotation algebra $\cA_{\theta}$ has K-theory $\bZ^2$ and any trace $\tau$ of $\cA_{\theta}$ evaluates as $\tau:(a,b)\in K_0(\cA_{\theta})=\bZ^2\rightarrow a+\theta b$ (see \cite{ELL84}). This condition on $\theta$ holds if and only if $e^{2\pi i \theta n^l}=1$ for some $l\in \bN$.

\end{remark}
We now turn to K-theoretic restrictions to the lifting obstructions of $G$-kernels that do not require the underlying C$^*$-algebra to have a trace. These are inspired by the restrictions to the lifting obstructions of $G$-kernel introduced in \cite{IZ23}. In what follows we will make use of the fact that whenever the unitary group $\cU(A)$ of a C$^*$-algebra $A$ is path connected, then it admits a universal covering group. While this is implicitly used in the literature (see, e.g., \cite{GIIZPE25}), we were unable to locate a convenient reference for a proof, so we include one for completeness.
\begin{lemma}\label{lem:semilocalsimplyconnected}
    Let $A$ be a unital C$^*$-algebra. Then $\cU(A)$ is semilocally simply connected. In particular, if $\cU(A)$ is connected it has a universal cover.
\end{lemma}
\begin{proof}
    Let $u\in \cU(A)$ and let $\cV$ be the ball around $u$ of radius $1/2$ in $\cU(A)$.  Let $\xi(t)$ be a loop in $\cV$ based at $v$. Then
\[
\|v^*\xi(t)-1\|<\|\xi(t)-u\|+\|u-v\|<1
\]
for all $t\in[0,1]$. Thus $v^*\xi(t)$ defines a unitary element in the C$^*$-algebra
\[
C=\{f\in C([0,1],A):f(0)=f(1)\in \bC1_A\}
\]
which is contained in the open ball of radius $1$ around the unit. Thus 
\[
v^*\xi(t)=e^{2\pi i h_1(t)}\ldots e^{2\pi i h_n(t)}
\]
for $h_i\in C^{sa}$ for $1\leq i\leq n $. Then the homotopy
\[
H(s,t)=ve^{2\pi i s h_1(t)}\ldots e^{2\pi i s h_n(t)}
\]
is a homotopy from $\xi(t)$ to the constant path at $v$. Hence $\cU(A)$ is semilocally simply connected.
\end{proof}

\par Let $A$ be a unital, separable C$^*$-algebra with a connected unitary group. Denote by $\widetilde{\cU}(A)$ the universal cover of $\cU(A)$ that has a concrete description as follows. Firstly denote by $C_{*}([0,1],A)$ the C$^*$-algebra of continuous paths in $A$ starting in $\bC$. Note that as $A$ is a separable then so is $C_{*}([0,1],A)$. Thus $\cU(C_{*}([0,1],A)$ is a Polish group. Denote its closed subgroup of unitary continuous paths starting at $1_A$ by $\cP_{\cU(A)}$ which is hence also a Polish group. Denote by $\Omega_{0}\cU(A)$ the subgroup of $\cP_{\cU(A)}$ consisting of the loops in $\cU(A)$ that are based homotopic to the constant loop at $1_A$. This is a closed subgroup of $\cP_{\cU(A)}$. The quotient group 
\begin{equation}
\widetilde{\cU}(A)=\cP_{\cU(A)}/\Omega_0\cU(A)
\end{equation}
is a model for the universal cover of $\cU(A)$ and it forms a Polish group. Consider the continuous surjection $\pi:\widetilde{\cU}(A)\rightarrow P\cU(A)$ defined by $\pi(f)=f(1)$ for a representative element $f\in \cP_{\cU(A)}$. We note that $\ker(\pi)$ coincides with the closed subgroup
\begin{equation}
    K_0^{\#}(A)=\{f\in \cU_{*}(C([0,1],A): f(1)\in Z\cU(A)\}/\Omega_0\cU(A).
\end{equation}
It is shown in \cite[Lemma 3.5]{GIIZPE25} that the exact sequence
\[
0\rightarrow K_0^{\#}(A)\xrightarrow{j} \widetilde{\cU}(A)\xrightarrow{\pi} P\cU(A)\rightarrow 0
\]
is a central extension of groups and $j$ and $\pi$ are continuous. The map $\pi$ induces a bijective continuous map $\widetilde{\cU}(A)/K_0^{\#}(A)\rightarrow P\cU(A)$. As both  $\widetilde{\cU}(A)/K_0^{\#}(A)$ and $P\cU(A)$ are Polish groups it follows from \cite[Theorem 14.12]{KE95} and Lemma \ref{lem:automaticcontinuity} that there is an isomorphism $\widetilde{\cU}(A)/K_0^{\#}(A)\cong P\cU(A)$.
\par Let $u$ be a projective representation of $G$ on $A$. Then by Lemma \ref{lem:lift} we may pick a Borel lift $\widetilde{u}:G\rightarrow \widetilde{\cU}(A)$ such that $\pi\circ\widetilde{u}=u$. Then as in Section \ref{subsec:projreps} the function $\lambda_{g,h}=\widetilde{u}_g\widetilde{u}_h\widetilde{u}_{gh}^*\in K_0^{\#}(A)$ is a Borel $2$-cocycle and its cohomology class is well-defined.
\begin{definition}
    Let $A$ be a unital, separable C$^*$-algebra with a connected unitary group and $u:G\rightarrow P\cU(A)$ be a projective representation. We denote the associated class in $H^2_b(G,K_0^{\#}(A))$ by $\widetilde{\ob}(u)$.
\end{definition}
As the projection map $\rho:\widetilde{\cU}(A)\rightarrow \cU(A)$ sending a representative function $f\in \widetilde{\cU}(A)$ to $f(1)$ is continuous, any Borel lift $\widetilde{u}$ to $\widetilde{U}(A)$ gives rise to a Borel lift $\rho(\widetilde{u})$ to $U(A)$. Thus 
\begin{equation}\label{eqn:onstructiontilde}
    \rho_*(\widetilde{\ob}(u))=\ob(u)
\end{equation}
Note that by \cite[Theorem 5.3]{GIIZPE25}, under the assumption that $A$ is unital, separable and such that the canonical map $\pi_1(\cU(A))\rightarrow K_0(A)$ is surjective, the obstruction $\widetilde{\ob}$ recovers $\ob_\tau$. Indeed, one has a surjection 
\[
p_2:\widetilde{\cU}(A)\cong \bR\times \widetilde{S\cU}_\tau(A)\rightarrow S\cU_\tau(A)
\]
by projecting onto the endpoint for a given path in $S\cU_\tau(A)$. The surjection $p_2$ is such that the diagram
\[
\begin{tikzcd}
    \widetilde{\cU}(A)\ar[r,"p_2"]\ar[d,"\pi"] & SU_\tau(A)\ar[ld]\\
    P\cU(A)
\end{tikzcd}
\]
commutes. Thus
\[
{p_2}_*(\widetilde{\ob}(u))=\ob_\tau(u)
\]
for any projective representation $u:G\rightarrow P\cU(A)$.
\begin{remark}\label{rmk:formofK0hash}
It is shown in the proof of \cite[Lemma 3.4 (i)]{IZ23} that whenever $A$ has trivial centre there is a continuous surjective homomorphism $\pi_1(\cU(A))\times \bR\rightarrow K_0^{\#}(A)$ with kernel $\bZ(e_1,-1)$ where $e_1$ denotes the loop $e_1(t)=e^{2\pi i t}$. Thus it induces a bijective continuous map
\[
\frac{\pi_1(\cU(A))\times \bR}{\bZ(e_1,-1)}\rightarrow K_0^{\#}(A)
\]
Now as both the range and source are Polish groups it follows that the inverse is also continuous and hence this is an isomorphism of topological groups (\cite[Theorem 9.10,Theorem 14.12]{KE95}). Under this identification the map $\rho$ becomes
\begin{align*}
\rho:\frac{\pi_1(\cU(A))\times \bR}{\bZ(e_1,-1)}&\rightarrow \bT\\
(x,\lambda)&\mapsto e^{2\pi i \lambda}
\end{align*}
Moreover, by \cite{DiscreteFG}, $\pi_1(\cU(A))$ is discrete as $\cU(A)$ is semilocally simply connected by Lemma \ref{lem:semilocalsimplyconnected}. 
\end{remark}
The invariant $\widetilde{\ob}$ allows us to characterise the possible values of $\ob$ for finite group projective representations on the Cuntz-algebras $\mathcal{O}_n$ for $2\leq n \leq \infty$. In the lemma below we denote by $\exp:\bR\rightarrow \bT$ the exponential map defined by $\exp(\lambda)=e^{2\pi i \lambda}$.
\begin{theorem}\label{thm:cuntz}
 Let $u:G\rightarrow P\cU(\mathcal{O}_{n+1})$ be a projective representation for $1\leq n< \infty$, then $\ob(u)\in nH^2_{b}(G,\bT)$. If $n=\infty$, then $\ob(u)$ is in the image of $\exp_*:H_b^2(G,\bR)\rightarrow H_b^2(G,\bT)$. Moreover, if $G$ is finite and $[\omega]\in nH^2(G,\bT)$ then there exists a $(G,[\omega])$-projective representation on $\cO_{n+1}$.
\end{theorem}
\begin{proof}
    Let $u:G\rightarrow P\cU(\cO_{n+1})$ for $n<\infty$ be a projective representation. The C$^*$-algebras $\mathcal{O}_{n+1}$ have connected unitary group by \cite{CU81}. Moreover by \cite{CU81} and \cite[Theorem 4.3]{THO91} there exists a pointed isomorphism $(\pi_1(\cU(\cO_{n+1})),e_1)\cong (\bZ_{n},[1])$. Thus it follows from Remark \ref{rmk:formofK0hash} that
    \[
    K_0^{\#}(\cO_{n+1})\cong \frac{\bZ_{n}\oplus \bR}{\bZ([1],-1)}\cong \bT
    \]
    with the second isomorphism of topological groups given by
    \begin{align*}
        \frac{\bZ_{n}\oplus \bR}{\bZ([1],-1)}&\rightarrow \bT\\
        ([k],\lambda)&\mapsto e^{\frac{2\pi i (\lambda+k)}{n}}.
    \end{align*}
    After identifying $K_0^{\#}(\cO_{n+1})$ with $\bT$ it follows from \eqref{eqn:onstructiontilde} that $\ob(u)$ is in the image of the map 
    \[
    H^2_{b}(G,\bT)\xrightarrow{\times n}H^2_{b}(G,\bT)
    \]
    as required.
    \par If instead $u:G\rightarrow P\cU(\mathcal{O}_\infty)$ then there exists a pointed isomorphism $(\pi_1(\cU(\cO_\infty)),e_1)\cong (\bZ,[1])$ and an isomorphism
    \begin{align*}
        K_0^{\#}(\cO_{\infty})\cong \frac{\bZ\oplus \bR}{\bZ(1,-1)}&\cong \bR\\
        (k,\lambda)&\mapsto k+\lambda.
    \end{align*}
    After identifying $K_0^{\#}(A)$ with $\bR$ it follows from \eqref{eqn:onstructiontilde} that $\ob(u)$ is in the image of the map 
    \[
        H^2_b(G,\bR)\xrightarrow{\operatorname{exp}_*} H^2_b(G,\bT).
    \]
    \par Now suppose that $G$ is finite and $[\omega]\in nH^2(G,\bT)$. Let $[\eta]\in H^2(G,\bT)$ such that $n[\eta]=[\omega]$ and $r$ be the order of $[\omega]$ in $H^2(G,\bT)$. Then  by \cite[III. Corollary 10.2]{BRO82}) we have that
    \[
    \frac{|G|}{\gcd(|G|,n)}[\omega]=\frac{n}{\gcd(n,|G|)}|G|[\eta]=0=|G|[\omega].
    \]
    So by definition $r$ divides both
    $\frac{|G|}{\gcd(|G|,n)}$ and $|G|$ and hence $\gcd(r,n)=1$. As $r$ and $n$ are coprime then $K_0(\mathcal{O}_{n+1})=\bZ_n$ is a uniquely $r$ divisible group and so $K_0(\cO_{n+1})\cong K_0(\cO_{n+1}\otimes M_{r^\infty})$ (see for example the proof of \cite[Corollary 5.4.4]{thesis:sergio}). It follows from the Kirchberg--Phillips classification theorem (\cite{phillipsclass},\cite{KI95}) that $M_{r^\infty}\otimes \mathcal{O}_{n+1}\cong \mathcal{O}_{n+1}$ and there exists a unital embedding 
    \[
    M_{r^\infty}\rightarrow \mathcal{O}_{n+1}.
    \]
    By Theorem \ref{thm:UHF} there exists a $(G,[\omega])$-projective representation on $M_{r^\infty}$ so there also exists one on $\mathcal{O}_{n+1}$.
\end{proof}
For finite groups $G$ the cohomology group $H^2(G,\bR)$ is zero. Thus we have the following immediate consequence of Theorem \ref{thm:cuntz}.
\begin{corollary}\label{cor:oinfinity}
    Let $G$ be a finite group. Then any projective representation of $G$ on $\cO_\infty$ lifts to a unitary representation on $\cO_{\infty}$. 
\end{corollary}
However, for infinite groups there are many non-trivial projective representations on $\cO_\infty$.
\begin{proposition}\label{prop: Oinftgrp}
    Let $[\omega]\in H^2(\bZ^n,\bT)$ then there is a $(\bZ^n,[\omega])$ representation on $\cO_\infty$.
\end{proposition}
\begin{proof}
    The C$^*$-algebra $A$=C$^*(\bZ^n,\omega)$ is a non-commutative $n$-torus as consequence of \cite[Lemma 5]{THO87}. The result now follows as every non-commutative $n$-torus admits a unital $^*$-homomorphism into $\cO_\infty$. Indeed, it is shown in \cite[Theorem 2.2]{ELL84} that $K_0(A)$ isomorphic to $\bZ^{2^{n-1}}$ with $[1_A]$ being identified with $(1,0,\ldots,0)$. Note that there is a pointed map of abelian groups $(\bZ^{2^{n-1}},(1,0,\ldots,0))\rightarrow (\bZ,1)$. Thus, as $A$ is nuclear and satisfies the universal coefficient theorem, it follows as a consequence of \cite[Theorem A]{GA24} that there is a unital $^*$-homomorphism $A\rightarrow \cO_\infty$. 
\end{proof}
On the opposing side to $\cZ$ we have $\cO_2$, for which there are no K-theoretic restrictions to lifting obstructions of projective representations. In fact any circle valued $2$-cocycle of a discrete exact group arises as the lifting obstruction of a projective representation on $\cO_2$. To show this we first require the following observation.
\begin{remark}\label{rmk:exactness}
    Whenever $G$ is a discrete exact countable group and $\omega\in Z^2(G,\bT)$ then $C_r^*(G,\omega)$ is exact. Indeed, to see this consider the associated central extension 
    \[
    N\rightarrow G\times_{\omega} N\rightarrow G
    \]
    where $N$ is the countable abelian subgroup of $\bT$ generated by the image of $\omega$ and $G\times_{\omega}N$ consists of pairs $(g,\lambda)$ for $g\in G$ and $\lambda\in N$ with the product
    \[
    (g,\lambda)(h,\mu)=(gh,\omega(g,h)\lambda\mu).
    \]
    Both $G$ and $N$ are exact and thus $G\times_{\omega} N$ is exact. As $C_r^*(G,\omega)$ is a quoteint of $C_r^*(G\times_{\omega} N)$ the C$^*$-algebra $C_r^*(G,\omega)$ is also exact.
\end{remark}
It follows from the $\cO_2$ embedding theorem (\cite{KI95}) and Remark \ref{rmk:exactness}  that for any discrete exact countable group $G$ and $\omega\in Z^2(G,\bT)$ there is a unital embedding 
\[
C^*_{r}(G,\omega)\rightarrow \cO_2.
\]
which gives the existence of a $(G,[\omega])$-representation on $\cO_2$. 
\begin{remark}\label{rmk:O2}
The example above illustrates that the K-theoretic restrictions to the lifting obstructions of projective representations lie in the position of $[1_A]$ inside $K_0(A)$. A unital Kirchberg algebra $A$ is said to be in \emph{Cuntz standard form} if  $[1_A]=0$ in $K_0(A)$ (or equivalently if there is a unital embedding of $\cO_2$ into $A$). Therefore, for any discrete exact countable group $G$, $\omega\in Z^2(G,\bT)$ and Kirchberg algebra $A$ in Cuntz standard form, there is a $(G,[\omega])$-representation on $A$. So, even though in the case of $\cO_\infty$ every projective representation of a finite group $G$ lifts to a genuine representation, its standard form $\cO_\infty^{st}$, which is Morita equivalent to $\cO_\infty$, admits $(G,[\omega])$-representations for any $[\omega]\in H^2(G,\bT)$.
\end{remark}

\section{A family of $2$-cocycles that do not arise on $\mathcal{R}$}

As argued in the introduction, any circle valued $2$-cocycle of a discrete amenable group arises as the lifting obstruction of a projective representation on $\cR$. Although K-theory does not cause any restrictions to the possible lifting obstructions for unitary representations on $\cR$, there are other restrictions of representation theoretic nature. If $G$ is a group admitting a universal cover, we will denote its universal cover by $\widetilde{G}$.
\begin{theorem}\label{thm:propertyT}
    Let $G=Sp(2n,\bZ)$ for $n\geq 2$, then there is a $2$-cocycle $[\omega]\in H^2(Sp(2n,\bZ),\bT)$ such that there is no $(G,[\omega])$ representation on $\cR$.
\end{theorem}
\begin{proof}
    Let $Q=Sp(2n,\bZ)$ for $n\geq 2$. We first argue that as a consequence of the theorem in \cite{DE79} one gets a central extension
    \begin{equation}\label{eqn:ext1}
    0\rightarrow\bZ_3\rightarrow G\rightarrow Q\rightarrow 0
    \end{equation}
    with $G$ not residually finite. Indeed, denoting by 
    \begin{equation}\label{eqn:ext2}
    0\rightarrow \bZ\langle a\rangle\rightarrow  E\rightarrow Q\rightarrow 0
    \end{equation}
    the pullback of the extension
    \[
    0\rightarrow \pi_1(Sp(2n,\bR))\cong \bZ\rightarrow \widetilde{Sp(2n,\bR)}\rightarrow Sp(2n,\bR)\rightarrow 0
    \]
    under the inclusion $\iota:Q\rightarrow Sp(2n,\bR)$. It is a corollary of \cite[Th{\'e}or{\`e}me]{DE79} (as noted in its abstract) that $a^2$ is contained in the kernel of any homomorphism from $E$ to a finite group. The quotient of extension \eqref{eqn:ext2} by $a^3$ is the desired extension \eqref{eqn:ext1}.  By standard group cohomology, there is a $2$-cocycle $\eta\in Z^2(Q,\bZ_3)$ such that $G\cong Q\times_{\eta}\bZ_3$ (using the notation of Remark \ref{rmk:exactness}). Recall that the centre of $Q$ is isomorphic to $\bZ_2=\{I_{2n},-I_{2n}\}$. As we have that $H^2(\bZ_2,\bZ_3)=0$ we may tweak $\eta$ by a coboundary to get a cocycle $\ti{\eta}:Q\rightarrow \bZ_3$ such that $\ti{\eta}|_{Z(Q)}=1$. Moreover, we still have an isomorphism $G\cong Q\times_{\tilde{\eta}}\bZ_3$ as $\eta$ and $\tilde{\eta}$ are cohomologous. Let $\omega=\iota_*\tilde{\eta}\in H^2(Q,\bT)$ with $\iota:\bZ_3\rightarrow \bT$ the inclusion. 
    \par We proceed to prove the theorem by contradiction. Suppose that there exists a $(Q,[\omega])$-representation 
    $u$ on $\cR$. We pick a set theoretic lift $\tilde{u}$ to $\cU(\cR)$ which satisfies
    \[\tilde{u}_g\tilde{u}_h=\tilde{\eta}(g,h)\tilde{u}_{gh}.
    \]
    As $Q$ is a linear group, there is a faithful unitary representation $v$ of $Q$ on $\cR$ whose image intersects $\bT$ only when evaluated at $Z(Q)$. Indeed, $Q$ is a group of matrices in $M_{2n}$ which embeds unitally into $\cR$. By replacing $\tilde{u}$ with $\phi(\tilde{u}\otimes v)$ for an isomorphism $\phi: \cR\overline{\otimes} \cR\rightarrow \cR$, we may assume that $\tilde{u}:Q\rightarrow \cU(\cR)$ is faithful and intersects $\bT$ only when evaluated at $Z(Q)$. The element $z=I_{2n}$ is of order $2$, so we have the identity
    \[
    \tilde{u}_z^2=\tilde{\eta}(z,z)=1
    \]
     and $\tilde{u}_z=-1$. Thus the mapping
    \begin{align*}
    \Phi:G\cong Q\times_{\tilde{\eta}} \bZ_3&\rightarrow \cU(\cR)\\
    (g,\lambda)&\mapsto \tilde{u}_g\iota(\lambda).
    \end{align*}
    is a faithful unitary representation. The group $Q$ has property (T) (see \cite[Section 1.7]{PropertyT}). As property (T) is preserved by extensions, $G$ also has property (T) (\cite[Section 1.7]{PropertyT}). Thus $\Phi$ is a faithful unitary representation of a countable property (T) group which is not residually finite into the unitary group of $\cR$. This gives a contradiction by \cite[Corollary 1.2]{KI94}.
\end{proof}
\bibliographystyle{plain}
\bibliography{projreps}
\end{document}